\documentclass{amsproc}
\usepackage{amsmath,amsthm,amsfonts,amssymb,amscd,amsbsy}

\newtheorem{theorem}{Theorem}[section]
\newtheorem{lemma}[theorem]{Lemma}
\newtheorem{problem}[theorem]{Problem}
\newtheorem{proposition}[theorem]{Proposition}
\newtheorem{corollary}[theorem]{Corollary}
\newtheorem{remark}[theorem]{Remark}

\theoremstyle{definition}

\newcommand{\hr} {\hookrightarrow}

\numberwithin{equation}{section}

\begin{document}
       \title[   classification  of  the   spaces of compact operators      ] {A complete  classification  of  the   spaces of compact operators on  $C([1,\alpha], l_{p})$  spaces,  $1<p< \infty$}
			 \author{Dale E.  Alspach}
      \author{El\'oi Medina Galego}

  \begin{abstract}   We  complete  the classification, up to isomorphism,
of  the spaces of compact operators on  $C([1, \gamma], l_{p})$ spaces,
$1<p< \infty$.  
In order to do this,
we classify, up to isomorphism,  the spaces of compact operators
${\mathcal K}(E, F)$,  where $E= C([1, \lambda], l_{p})$ and   $F=C([1,
\xi], l_q)$  for  arbitrary ordinals    $\lambda$ and $\xi$ and    $1<
p \leq q< \infty$.

More precisely, we prove  
that it is relatively consistent with 
ZFC that  for any infinite  ordinals $\lambda$, $\mu$, $\xi$ and $\eta$ the following statements are  equivalent:

\begin{enumerate}
\item[(a)] ${\mathcal K}(C([1, \lambda], l_{p}), C([1, \xi], l_{q})) $ is
isomorphic to  ${\mathcal K}(C([1, \mu], l_{p}), C([1, \eta], l_{q})) .$

\item  [(b)] $\lambda$ and $\mu$ have the same cardinality   and
$C([1,\xi])$ is isomorphic to $C([1, \eta])$ or there exists an
uncountable regular  ordinal  $\alpha$  and $ 1 \leq m, n  < \omega$
such that $C([1, \xi])$ is isomorphic to $C([1, \alpha m])$ and $C([1,
\eta])$ is isomorphic to $C([1, \alpha n])$.  \end{enumerate}

Moreover, in ZFC,  if $\lambda$ and $\mu$ are finite ordinals and $\xi$ and
$\eta$ are infinite ordinals then the
statements (a) and (b') are equivalent.
\begin{enumerate}
\item  [(b')] $C([1,\xi])$ is isomorphic to $C([1, \eta])$ or there exists
an
uncountable regular  ordinal  $\alpha$  and $ 1 \leq m, n  \le \omega$
such that $C([1, \xi])$ is isomorphic to $C([1, \alpha m])$ and $C([1,
\eta])$ is isomorphic to $C([1, \alpha n])$. 
\end{enumerate}

    \end{abstract}
    
    \maketitle
    

    
    
    \par

    \section{Introduction}

We use standard  set theory and   Banach space theory  terminology and notions  as can be found 
in \cite{J} and  \cite{JL}, respectively.  Let $K$ be  
a compact Hausdorff space and X a  Banach space.  The space $C(K,X)$ 
denotes the Banach space of all 
continuous $X$-valued functions defined on $K$ equipped with the supremum
norm, i.e., $\|f\|=\sup_{k\in K} \|f(k)\|_X.$ We will write $C([1,
\alpha], X)=C(\alpha, X)$ when $K$ is the interval of ordinals $[1,
\alpha]=\{\xi: 1 \leq \xi \leq \alpha\}$ endowed with the order
topology. These spaces will be denoted by $C(K)$ and $C(\alpha)$,
respectively, in the case  $X=\mathbb R$.

For a set $\Gamma$,  $c_{0}(\Gamma, X)$ is the Banach space of all
$X$-valued functions $f$ on $\Gamma$  with the property that
for every $\epsilon>0$, the set $\{\gamma \in \Gamma:    \   \|f(\gamma)\|
\geq \epsilon \}$ is finite, and equipped with the sup norm.  This space
will be denoted by $c_{0}(\Gamma) $ in the case  $X=\mathbb R$ and  by
$c_{0}$  when, in addition,  the cardinality of  $\Gamma$ (denoted by
$|\Gamma|$) is $\aleph_{0}$.

Given 
Banach spaces $X$ and $Y$, ${\mathcal K}(X, Y)$ denotes the Banach space of 
compact operators from $X$ to $Y$. As usual, when $X=Y$, this space will be denoted by   ${\mathcal K}(X)$.    We write $X \sim Y$ when  
$X$ and $Y$ are isomorphic and $X \hookrightarrow Y$ when $Y$ contains  a  copy of $X$, that is, a subspace  isomorphic to $X$.

In this paper, we are mainly interested in   completing   the isomorphic 
classification of  the  spaces  ${\mathcal K}(C(\lambda, l_{p}), C(\xi, 
l_{q}))$  obtained partially   in \cite{G3}, \cite{G4} and  \cite{G5},  where $\lambda$ and $\xi$ are  arbitrary ordinals  and $l_p$ and $l_{q}$ are the classical  Banach spaces  of  scalar sequences  with $1 < p, q< \infty$.

For $p>q$, an isomorphic 
classification of these spaces was accomplished   in \cite[Remark 4.1.3]{G3} for 
$\lambda< \omega$ and $\xi \geq \omega$, and in \cite[Remark 1.7]{G4} for $\lambda \geq 
\omega$ and $\xi \geq \omega$. We pointed out that,  in both these cases,
it was crucial the following geometric  property  of the spaces of compact
operators from $l_{p}$ to $l_{q}$  \cite[ Main Theorem]{S2}.
$$c_{0} \not \hookrightarrow {\mathcal K}(l_
{p}, l_{q}). \eqno(1) $$
Thus, in the present work,  we turn our attention to the case $1<
p \leq q< \infty$. In this case the situation in quite
different. Indeed, in contrast to (1), by \cite[Theorem 5.3]{AG}, 
${\mathcal K}(l_ {p}, l_{q})$ is isomorphic to its $c_{0}$-sum
or, what is the same, that $${\mathcal K}(l_ {p}, l_{q}) \sim {\mathcal
K}(l_{p}, c_{0}(l_{q})). \eqno(2)$$
Nevertheless,  in   \cite{G5}   it was shown that the  following
cancellation law holds.
$${\mathcal K}(C(\lambda, l_{p}), C( \xi, l_{q})) \sim {\mathcal K}(C(
\mu, l_{p}), C(\eta, l_{q})) \Longleftrightarrow C(\xi) \sim C(\eta),$$
whenever $\lambda$, $\mu$, $\xi$ and $\eta$ are  infinite countable
ordinals or $\lambda \mu < \omega$ .

The main goal of this paper is  to extend  the above cancellation law by giving a complete isomorphic classification  of the spaces  ${\mathcal K}(C(\lambda, l_{p}), C(\xi, l_{q}))$, where  $\lambda$ and $\xi$ are arbitrary ordinals.

To do this,  first in section \ref{Kc0}  we state a  new  geometric property of the spaces of compact operators involving the spaces  $l_{p}$,  $l_{q}$ and  $c_{0}$.  Namely.
$${\mathcal K}(c_{0})    \not  \hookrightarrow {\mathcal K}(c_{0}(l_
{p}), l_{q})). $$
Moreover, in order to use the isomorphic classification of  some  $C(\xi,
X)$  spaces obtained in \cite{G21}, in section \ref{c0gamma} we prove a stability property of certain spaces of compact operators containing copies of  some $c_{0}(\Gamma)$ spaces. Specifically, we  show that  if $X$ is a Banach space and $\Gamma$ is a set of cardinality $\aleph_{1}$, then
$$c_{0}(\Gamma) \hookrightarrow {\mathcal K}(l_{p}, X)   \Longleftrightarrow                c_{0}(\Gamma) \hookrightarrow X.$$
In section \ref{classification}  we  present our main result  (Theorem \ref{iso}). It provides a necessary and sufficient  condition for two  spaces  ${\mathcal K}(C (\lambda,   l_{p}), C(\xi, l_{q}))$   to be isomorphic, whenever  $\lambda$ and $\xi$ are infinite and  the cardinality of $\lambda$   is strictly less than the least real-valued measurable cardinal $m_{r}$.

We recall that a cardinal number $m$ is a real-valued measurable cardinal if there exists a non-trivial  real-valued measure defined on all subsets of a set of cardinal $m$ for which points have measure $0$ \cite[ page 560]{E}.
	
Theorem \ref{iso}   is  a complete isomorphic classification of the    ${\mathcal K}(C(\lambda, l_{p}), C(\xi, l_{q}))$ spaces, where   $\lambda$ and $\xi$  are infinite. Indeed, it is well known that the existence of real-valued  measurable cardinals 
cannot 
be
proved  in ZFC \cite[pages 106 and 108]{KM}. On the other hand, it is relatively
consistent
with ZFC that real-valued measurable  cardinals do not exist \cite[Theorem 4.14, page
972]{GP}.
So it is relatively consistent with ZFC that  we obtain  a
complete isomorphic classification of the spaces ${\mathcal K}(C(\lambda, l_{p}), 
C(\xi, l_{q}))$, where $\lambda$ and  $\xi$ infinite.

Finally, in Theorem \ref{t2} we consider the remaining case  $\lambda$ is finite.  Then, by using (2) we complete the proof of the isomorphic classification  mentioned in the abstract.

\section{Preliminaries}

In this section we recall the isomorphic classification of the $C(\xi)$ spaces,
the generalization proved in \cite{G21}, and some properties of the spaces of
compact operators.

From now on  $|\xi|$  denotes the cardinality of the ordinal $\xi$. Recall
that an ordinal $\alpha$ is said to be regular if its cofinality is itself.
Otherwise $\alpha$ is a singular ordinal.   The smallest two infinite
regular  ordinals are $\omega$ and $\omega_{1}$.  We denote by
$X {\hat{\hat\otimes}} Y$ the injective tensor product of the Banach
spaces $X$ and $Y$.  

Bessaga and Pelczynski \cite{BP} showed that for infinite countable
ordinals, $\xi  \leq \eta$, $C(\xi)$ is isomorphic to  $C(\eta)$ if and only if
$\eta < \xi^\omega.$ In 1976 the classification was extended to nondenumerable ordinals $\xi$ and $\eta$,
\cite{GO} and \cite{Kis}, independently. The general criterion is the same
as that obtained by Bessaga and Pelczynski except in the case that   the initial ordinal $\alpha$ of cardinality $|\xi|=|\eta|$ 
is a nondenumerable, regular ordinal, and $\alpha\leq \xi \leq   \eta \leq \alpha^2.$ For
this case,  equality of cardinalities of the ordinal quotients, $\xi',\, \eta',$
where $\xi=\alpha \xi'+\delta,\, \eta=\alpha \eta'+\gamma$, and $\delta,\,
\gamma<\alpha$ is the requirement for isomorphism of $C(\xi)$ and
$C(\eta).$

For the spaces $C(\xi,X)$ and $C(\eta,X)$, where $X$ is an infinite dimensional Banach
space some additional difficulties arise in determining the isomorphic
classification. Indeed,  since  that  $C(\xi,X)$ is isomorphic to $C(\xi) {\hat{\hat\otimes}}
X$, it follows that $$C(\xi)\sim C(\eta)    \Longrightarrow   C(\xi,X)\sim C(\eta,X).$$
Moreover,  $$X\sim X \oplus X    \Longrightarrow     C(\alpha n,X)\sim C(\alpha,X)$$
 for all
positive finite ordinals $n$. Thus for a nondenumerable regular ordinal
$\alpha$ the classification is less fine than for the spaces of real-valued
functions. However, for some Banach spaces it is possible to determine
the isomorphic classification of  the spaces $C(\xi,X)$.  

To state this result recall that a Banach space $X$ is said to have the Mazur Property,
MP in short, if every element of $X^{**}$ which is sequentially weak*
continuous is weak* continuous and thus is an element of $X$. Such
spaces were investigated in \cite{E}  and   \cite{Ka} and sometimes
they are also called d-complete \cite{L} or $\mu$B-spaces \cite{W}. 
In \cite{Ka} it was shown (Theorem 4.1)
that $C(\xi)$ for $\xi<\omega_1$, $\left(\sum_{\gamma\in \Gamma}
X_\gamma\right )_{l_1},$ where $|\Gamma|$ is non-measurable and $X_\gamma$
has MP for all $\gamma,$ (Theorem 3.1), and $X{\hat{\hat\otimes}}Y$,
where $X$ has MP and $Y$ is separable, (Corollary 5.2.3) have MP.

Let $\mathcal F$ be the class of Banach  spaces containing no copy of
$c_{0}(\Gamma)$, where $|\Gamma|= \aleph_{1}$, and  having the Mazur
property. For this class
the main theorem of \cite{G21} gives a classification of the spaces
$C(\xi,X).$ 

\begin{theorem} \label {Bu}Let $X \in \mathcal F$ and 
$\alpha$  an initial ordinal and $\xi \leq \eta$ 
infinite  ordinals. Then 
\begin{enumerate}
\item If $C(\xi, X)$ is isomorphic to  $C(\eta, X)$ then
$|{\xi}| = |{\eta}|.$

\item Suppose  that $|\xi|=|\eta|=|\alpha|$ and assume that   
$\alpha$ is a singular ordinal,  or  
$\alpha$ is a nondenumerable regular ordinal with $\alpha^2 \leq
\xi$.  
Then $$C(\xi, X) \sim C(\eta, X) \Longleftrightarrow \eta
< \xi^{\omega}.$$
\item \label{regular_ord}
Suppose that $\alpha$ is an uncountable  regular ordinal,
$\xi, \eta \in [\alpha, \alpha^{2}]$ and let $\xi', \eta', \gamma$
and  $\delta$ be ordinals
such that $\xi=\alpha \xi' + \gamma$, $\eta=\alpha \eta' + \delta$, 
$\xi', \eta'\leq \alpha$ and  $\gamma, \delta <\alpha$. Then 
$$C(\xi, X) \sim C(\eta, X),$$    
if and only if one of the following
statements holds:
\begin{enumerate}
\item[(a)]\label{finite_quotient} $|{\xi'}| \ |{\eta'}| \leq
\aleph_{0}$ and $c_{0} (J, X)$ is isomorphic to  
$c_{0}(I, X)$, where $J$ and $I$ are sets with $|J|=|{\xi'}|$ and $|I|
=|{\eta'}|$.
\item[(b)] $|{\xi'}| = |{\eta'}| > \aleph_{0}$.
\end{enumerate}
\item  Suppose that $\alpha$ is an uncountable  regular ordinal and  
$\alpha \leq \xi < \alpha^2 \leq \eta$. Then
$$C(\xi, X)  \not \sim C(\eta, X).$$ 
\end{enumerate}
\end{theorem}

\begin{remark} \label{rrr}
Notice that for a fixed $X \in \mathcal F$,
the isomorphic classification of the spaces $C(\xi,X)$ 
is the same as that of spaces $C(\xi)$ except in the case
(\ref{regular_ord}) of the theorem. As noted above this
case can occur when $X \sim X\oplus X.$
\end{remark}

We are concerned with spaces of compact operators $\mathcal
K(C(\lambda,Y),C(\xi,Z))$, where $Y$ and $Z$ are Banach spaces 
which are isomorphic their squares, i.e, $Y \sim Y\oplus Y$, and such that
their duals have the approximation property,
and $\lambda$ and $\xi$ are ordinals.

 Because
$C(\lambda, Z)$  has the approximation property if $Z$ has the
approximation property, by \cite[Proposition 5.3]{DF} we see that   $$\mathcal
K(C(\lambda,Y),C(\xi,Z)) \sim C(\lambda,Y)^* {\hat{\hat \otimes}} C(\xi,Z).$$
Moreover, notice that $C(\lambda,Y)^{* }$ is isomorphic to  $l_1(\Gamma,Y^*),$ where $\Gamma$ is a set
with cardinality $|\lambda|.$ Thus the isomorphism class of the spaces  $\mathcal
K(C(\lambda,Y),C(\xi,Z))$ can be represented by $\mathcal
K(C(\lambda_0,Y),C(\xi,Z))$ where $\lambda_0$ is the smallest ordinal with
$|\lambda_0|=|\lambda|$ or if $\lambda$ is finite, then $\lambda_0=1$.

In \cite{AG} it was shown that $$C(\omega,
\mathcal K(l_p,l_q)) \sim \mathcal K(l_p,l_q),$$ for $1< p\le q <\infty.$
Because $$C(\omega,
\mathcal K(l_p,l_q)) \sim \mathcal K(l_p,C(\omega,l_q)),$$ thus in the case
$Y=l_p$ and $Z=l_q$ there is the possibility of additional collapsing of
the isomorphism classes.
Indeed, if $\alpha$ is either a finite or uncountable regular ordinal and
$\xi \le \omega$. then
the spaces $C(\alpha \xi)$ are isomorphically distinct. However, if
$\lambda <\omega,$ then 
$$\mathcal K(C(\lambda,l_p),C(\alpha \xi,l_q)) \sim
l_{p'}{\hat{\hat{\otimes}}} C(\alpha \xi) {\hat{\hat{\otimes}}} l_q  \sim C(\alpha, C(\xi, l_{p'}{\hat{\hat{\otimes}}} l_q))  \sim
C(\alpha,
l_{p'}{\hat{\hat{\otimes}}} l_q), $$
and therefore
$$\mathcal K(C(\lambda,l_p),C(\alpha \xi,l_q)) \sim \mathcal K(l_p,C(\alpha,l_q)).$$

The next proposition summarizes these remarks and the consequences
of Theorem \ref{Bu}. Below, if $\gamma$ is an ordinal, $\gamma_0=1$ if
$\gamma<\omega$ and otherwise $\gamma_0$ is the smallest ordinal with the
same cardinality as $\gamma.$ 

\begin{proposition}\label{canonical_ord} Suppose that $Y$ and $Z$ are
Banach spaces in class $\mathcal
F$, have the
approximation property, are isomorphic to their squares, and $\mathcal
K(Y,Z)$  is isomorphic to  $C(\omega,\mathcal
K(Y,Z)).$ Then $$\mathcal
K(C(\lambda,Y),C(\xi,Z)) \sim \mathcal
K(C(\lambda_0,Y),C(\psi(\lambda,\xi),Z)),$$ where
\begin{enumerate}
\item{} if $\lambda \ge \omega,$ $\xi_0$ is an uncountable regular ordinal, $\xi \le
\xi_0^2$, and $\xi=\xi_0 \xi' +\delta$, $\delta<\xi_0$ then $\psi(\lambda,\xi)=\xi_0\xi'_0$
if $\xi'\ge \omega$ and $\psi(\lambda,\xi)=\xi_0$ if $\xi'<\omega$.
\item{} if $\lambda <\omega,$ $\xi_0$ is an uncountable regular ordinal,
$\xi \le
\xi_0^2$, and $\xi=\xi_0 \xi' +\delta$, $\delta<\xi_0$ then
$\psi(\lambda,\xi)=\xi_0\xi'_0$
if $\xi'> \omega$ and $\psi(\lambda,\xi)=\xi_0$ if $\xi'\le \omega$.

\item{} $\xi_0$ is an uncountable regular ordinal, and $\xi_0^2<\xi$, then
$\psi(\lambda,\xi)=\max\{ \xi_0^2,\gamma\}$ where $\gamma$ is the smallest  
ordinal such that $\gamma^\omega>\xi$ 
\item{} if $\xi_0<\omega$ or $\lambda<\omega$ and $\xi<\omega^\omega$, $\psi(\lambda,\xi)=1$
\item{} if $\xi_0=\omega$ and $\lambda \ge \omega$ or
$\xi_0$ is an infinite singular ordinal,
$\psi(\lambda,\xi)$ is the smallest
ordinal $\psi$ such that $\psi^\omega>\xi.$
\end{enumerate}
\end{proposition}

Our goal in the next few sections is to show that
the  $\mathcal K(C(\lambda_0,Y),C(\psi(\lambda,\xi), Z))$  spaces are from
distinct isomorphism classes if $Y$ is $l_p$, and $Z$ is $l_q$
with $1<p\le q<\infty.$

\begin{remark} \label{card} Observe that if $\lambda$ and $\xi$ are infinite ordinals,
then $|\lambda|=|\lambda_0|$ and $|\xi|=|\psi(\lambda,\xi)|.$
\end{remark}

\section{ On the subspaces of   ${\mathcal K} (c_0(X),
l_{q}(Y))$ spaces,   $1\leq q <\infty$  }\label{Kc0}

The main aim of this section is to prove   Theorem \ref{no_c0_l1}    which
is  the key ingredient for completing the isomorphic classification of the
spaces of compact operators considered in this work.  To do this,  we need
an auxiliary result. Let  $X$ and $Y$ be Banach spaces and $T\in{\mathcal
K} (c_{0}(X), l_{q}(Y))$, with   $1\leq q <\infty$.     Represent $T$
as a matrix with entries in ${\mathcal K} (X,Y)$.  By \cite[Proposition
1.c.8 and following Remarks]{LT} for any blocking of the $c_0$-sum and
$l_q$-sum the operator given by the block diagonal of
the matrix is a bounded linear operator with norm no larger than $\|T\|$.
The next result gives a little more information about $\|T\|.$

\begin{lemma}\label{c0sum} Let $X$ and $Y$ be Banach spaces and let $1\leq
q< \infty$.    Then the norm of a block diagonal operator of an operator in
${\mathcal K} (c_{0}(X),
l_{q}(Y))$ is the $l_q$ norm of the
norms of the operators on the blocks and the block diagonal operator is compact.
\end{lemma}

\begin{proof}  
Let $D_j$ be the $j$th block of
the block diagonal operator $D$ and $(x_j) \in c_{0}(X)$ such that for each
$j,$ $x_j$ is supported block corresponding to $D_j.$ Then
\begin{align*}
\|(D_j x_j)\|_{l_{q}(Y)}=\left ( \sum \|D_j x_j\|_Y^q\right )^{1/q} &\leq\left ( \sum
\|D_j\|^q \|x_j\|_X^q\right )^{1/q} \\ &\le (\sup_j \|x_j\|_X) \|(D_j)\|_{l_q}.  \end{align*}    
Suppose $1>\epsilon>0,$ $x_j \in X$, supported in the $j$th block , with $\|x_j\|_X=1$ and $$\|D_j
x_j\|\ge (1-\epsilon)\|D_j\|,$$ $j=1,2,\dots.$ For each $N$, let
$z_N=(x_1,x_2,\dots,x_N,0,0,\dots).$ Then
\begin{align*}
\|D z_N\|=\|(D_j x_j)_{j\le N}\|_{l_{q}(Y)}
&=\left ( \sum_{j\le N}
\|D_j x_j\|_Y^q\right )^{1/q}\\ &\ge \left ( \sum_{j\le N}
(1-\epsilon)\|D_j\|^q\right )^{1/q}
\end{align*}
Taking the supremum over $\epsilon$ and $N,$ gives $$\|D\|\ge
\|(D_j)\|_{l_q}.$$

It remains to show that the diagonal of a compact operator is compact. Let
$Q_n$, respectively, $R_n$, be the projection onto the first $n$
coordinates of $l_q(Y),$ respectively, $c_0(X)$. A
gliding hump argument shows that for any $T \in {\mathcal K} (c_{0}(X),
l_{q}(Y))$ and $\epsilon>0$ there is an $N$ such that for all $n \ge N,$
$$\|(I-Q_n)T\|<\epsilon.$$ Observe that for $j>n$,
$$(Q_j-Q_{j-1})T(R_j-R_{j-1})=(Q_j-Q_{j-1})(I-Q_n)T(R_j-R_{j-1}).$$
Thus for $j>n$, the diagonal
element of $T$, $D_j$, is also the diagonal element of $(I-Q_n)T$.
By another application of Tong's result \cite[Proposition
1.c.8 ]{LT} it follows that
$$\|(D_j)_{j<\infty}-(D_1,D_2,\dots,D_N,0,0,\dots)\|<\epsilon.$$ Thus the
diagonal of $T$ is the limit of compact operators and therefore is compact.

\end{proof}

The lemma will be applied near the end of the proof of the next result with
$X=l_p$ and $Y=l_q.$

\begin{theorem} \label{no_c0_l1} Suppose that $1<p\le q<\infty.$ Then
$${\mathcal K}(c_{0}) \not \hr {\mathcal K}(c_{0}(l_p),l_{q}).$$
\end{theorem}

\begin{proof} First of all notice that ${\mathcal  K}(c_{0})$ is isomorphic
to  $$l_i{\hat{\hat{\otimes}}}c_0=c_{0}(l_1)=(\sum_{n=1}^\infty l_1)_{c_0} .$$ 
Let $P_n$ denote the standard projection from  $c_{0}(l_p)$ onto
$(\sum_{k=1}^n l_p)_{\infty}$ and let 
$(Q_n)$
be the basis projections for
$l_q$.

Suppose that $(T_{s,n})$ is a normalized sequence in ${\mathcal
K}(c_{0}(l_p),l_{q})$ which is equivalent to the unit vector basis of
$(\sum_{n=1}^\infty l_1)_{c_0}$, i.e., there are $\delta>0$ and $K<\infty$ such that
$$\delta \max_s \sum_n|a_{s,n}| \le \|\sum_s \sum_n a_{s,n} T_{s,n}\|\le K
\max_s \sum_n|a_{s,n},|$$ for
all finitely non-zero sequences $(a_{s,n}).$

For fixed $s,r$, $(T_{s,n} P_r)$ is (equivalent to) a sequence in ${\mathcal
K}(\sum_{n=1}^r l_p,l_q),$ which is isomorphic to ${\mathcal 
K}(l_p,l_q),$ and thus by    \cite[Theorem 3.1]{S1}           cannot be equivalent
to the unit vector basis of $l_1.$   Hence there exists a normalized
blocking of
$(T_{s,n})_n$, $(U_{s,m}),$  and an increasing sequence $(r_{s,m})_m$ such that 
$$\|U_{s,m}P_{r_{s,m}}\|\rightarrow 0.$$  Further, because $U_{s,m}$ is
compact, for fixed $s,m$,
$$\|U_{s,m}(I-P_{r})\|\rightarrow 0.$$ Thus by a perturbation argument and
passing to subsequences we may assume that $(r_{s,m})$ is strictly
increasing in some order and that
$$U_{s,m}(P_{r_{s,m}}-P_{r_{s',m'}})=U_{s,m},$$ for all $s,m,$ where
$r_{s',m'}$ is the maximal element strictly smaller than $r_{s,m}.$

For fixed $s,m$ $$\|(I-Q_k)U_{s,m}\|\rightarrow 0,$$ 
because $q<\infty$
and $U_{s,m}$ is compact.
If $(s_t,m_t)$ is a
sequence with $(s_t)$ strictly increasing $(U_{s_t,m_t})$ is equivalent to
the usual unit vector basis of $c_0$. 
For fixed $k$ we claim that $$\|Q_k U_{s_t,m_t}\|\rightarrow 0.$$ Suppose not
and let $x_t \in
c_0(l_p)$, $\|x_t\|=1$ and $$\|Q_k U_{s_t,m_t}x_t\|>\|Q_k U_{s_t,m_t}\|/2,$$
for all $t$. By
passing to a subsequence and a perturbation argument
because $k$ is fixed, we may assume that $(Q_k U_{s_t,m_t}x_t)$ is independent of $t$. We may
also assume by the choice of $r_{s,m}$
that $$(P_{r_{s_t,m_t}}-P_{r_{s'_t,m'_t}})x_t=x_t,$$ for all $t$.
Thus for any finite $j$, $\|\sum_{i=1}^j x_i\|=1$ but $$K\ge \|\sum_t Q_k
U_{s_t,m_t}\sum_{i=1}^j x_i\|=\|\sum_{t=1}^j Q_kU_{s_t,m_t}x_t\|>j\|Q_k
U_{s_1,m_1}\|/2.$$ This contradicts $K<\infty$ and establishes the claim.

By again passing to a suitable subsequence and a perturbation
we can assume that we have
$(U_{s_t,m_t})$ with $(s_t)$ strictly increasing and $(k_t)$ strictly
increasing such that $$U_{s_t,m_t}=(Q_{k_t}-Q_{k_{t-1}})U_{s_t,m_t}.$$
However by Lemma \ref{c0sum} this implies that
$(U_{s_t,m_t})_t$ is equivalent to the basis of
$l_q$.
\end{proof}
\begin{remark} The proof of Theorem \ref{no_c0_l1} uses $l_p$ and $l_q$ in
a very minor way. $l_p$ could be replaced by a space $Y$ which isomorphic to
its square, and $l_q$ by a space with a basis that has a non-trivial lower
estimate such that $l_1$ is isomorphic to no subspace of  $\mathcal K(Y,Z).$
\end{remark}

\begin{corollary}\label{notiso} If $\lambda,\xi \ge \omega$, and $1<p\le q<\infty$, then
$\mathcal K(C(\lambda,l_p),C(\xi,l_q))$ is not isomorphic to a subspace of
$\mathcal K(C(\lambda,l_p),l_q).$
\end{corollary}

\begin{proof}  Observe that $ \mathcal K(C(\lambda,l_p),C(\xi,l_q))$ is isomorphic to
$$l_1(\lambda,l_{p'}){\hat{\hat{\otimes}}}C(\xi,l_q)\sim
l_1(\lambda,l_{p'}){\hat{\hat{\otimes}}}C(\xi){\hat{\hat{\otimes}}}l_q \sim
C(\xi, l_1(\lambda,l_{p'}){\hat{\hat{\otimes}}}l_q).$$ 
Thus, it follows that  $c_0(l_1)$ is
isomorphic to a subspace of $ \mathcal K(C(\lambda,l_p),C(\xi,l_q))$.
If $ \mathcal K(C(\lambda,l_p),C(\xi,l_q)$ is isomorphic to a subspace of
$\mathcal K(C(\lambda,l_p),l_q)$ then $c_0(l_1)$ is also. Because $c_0(l_1)$ is
separable, this would imply that $c_0(l_1)$ is isomorphic to a subspace of
$\mathcal K(C(\mu,l_p),l_q)$ for some countable ordinal $\mu$ and hence to
$\mathcal K(C(\omega,l_p),l_q)$, contradicting  Theorem \ref{no_c0_l1}.
\end{proof}

\section{Copies of $c_{0}(\Gamma)$ in $\mathcal K(l_{p}, X)$ spaces, $1<p<
\infty$}\label{c0gamma}

Before proving our main theorem  we need  another preliminary result which
will yield additional spaces that are in the class $\mathcal F$ and thus
Theorem \ref{Bu} will be applicable. In particular we will need to know
that
the spaces $\mathcal K(l_{p}, X)$, with $1<p< \infty$,
contain a copy of  $c_{0}(\Gamma)$ with $|\Gamma| =\aleph_{1}$ if
and only if  the Banach space $X$ has the same property. This is an
immediate consequence of the following proposition.
Below $\mathcal B(Y,X)$ is the space of bounded
linear operators from $Y$ to $X$.

\begin{proposition} \label{inj}If $X$ and $Y$ are Banach spaces and $Y$ is separable,
then $c_0(\Gamma)$, where $|\Gamma|=\aleph_1$, is isomorphic to a subspace
of $\mathcal B(Y,X)$ if and only if $c_0(\Gamma)$ is isomorphic to a
subspace of $X$.
\end{proposition}

\begin{proof}
Suppose that $(T_\gamma)_{\gamma \in \Gamma}\subset \mathcal B(Y,X)$  is
the image of the coordinate vectors $(e_\gamma)$ in $c_0(\Gamma)$ under some
isomorphism $S$. Let $(y_j)_{j=1}^\infty$ be a sequence which is dense in the
unit sphere of $Y$. For $n,j \in \mathbb N$ let $$N_{j,n}=\{\gamma \in
\Gamma: \|T_\gamma(y_j)\|> 1/n\}.$$ Since $\Gamma= \cup_{j,n} N_{j,n} $  there
is some $j_0,n_0$ so that $|N_{j_0,n_0}|=\aleph_1.$ For each $\gamma \in
N_{j_0,n_0}$ let $x_\gamma^* \in X^*$ such that $\|x_\gamma^*\|=1$
and $$S^*(y_{j_0}\otimes x_\gamma^*)(e_\gamma)=T^*_\gamma x_\gamma^*
(y_{j_0})=x_\gamma^*(T_\gamma(y_{j_0})>1/{n_0}.$$ By Rosenthal's
disjointness lemma \cite[Theorem 3.4]{R} and the remark following, 
there is a subset $N'$ of
$N_{j_0,n_0}$ of the same cardinality such that $(S^*(y_{j_0}\otimes
x_\gamma^*))_{\gamma\in N'}$ is equivalent to the standard unit  vectors in
$l_1(N)$ and norms $[e_\gamma:\gamma\in N']$. Therefore the mapping
$S_1:[e_\gamma:\gamma\in N'] \rightarrow X$ defined by
$S_1(e_\gamma)=T_\gamma(y_{j_0})$ extends to an isomorphism into $X$.

The converse is obvious.
\end{proof}

\begin{corollary}\label{class_F}
Let  $V$ and $W$ be Banach spaces such that $c_0$ is isomorphic to no subspace of  $V$, $V$ has MP, and
$W^*$ is separable and has the approximation property. If $|\lambda|$ is
non-measurable then 
$\mathcal K(W,l_1(\lambda,V))$ is in the class $F$.
\end{corollary}

\begin{proof}
Because $V$ does not contain subspace isomorphic to  $c_0$, it follows from a gliding hump argument
that $l_1(\lambda,V)$ has the same property.
By Proposition \ref{inj}, $c_0(\Gamma)$ is not isomorphic to a subspace of
$\mathcal K(W,l_1(\lambda,V))$. To see that $\mathcal K(W,l_1(\lambda,V))$
has MP we note that $\mathcal K(W,l_1(\lambda,V))$ is isomorphic to  
$W^*{\hat{\hat{\otimes}}}l_1(\lambda,V)$ by \cite[Proposition 4.3]{DF}. By the
results of Kappeler \cite{Ka}, $l_1(\lambda,V)$ has MP and consequently 
$W^*{\hat{\hat{\otimes}}}l_1(\lambda,V)$ has MP.
\end{proof}

\section{ A classification of   ${\mathcal K}(C (\lambda,   l_{p}), C(\xi,
l_{q}))$ spaces, $1< p \leq q < \infty$} \label{classification} In this
section we  classify, up to isomorphism,  the spaces of compact operators
${\mathcal K}(C (\lambda,   l_{p}), C(\xi, l_{q}))$, with  $1< p \leq q <
\infty$. 

The following theorem gives necessary conditions for
two  ${\mathcal K}(C (\lambda,   l_{p}), C(\xi, l_{q}))$  spaces to be
isomorphic in terms of the notation introduced in Proposition
\ref{canonical_ord} and provides a converse to that result for $Y=l_p$ and
$Z=l_q.$
Recall that  the density character of
a Banach space $X$ is the smallest cardinal number $\delta$ such that
there exists a set of cardinality $\delta$ dense in $X$.

\begin{theorem} \label{Nec} Suppose that  $1 <p \leq q < \infty$  and
$\xi$,  $\eta$,  $\lambda$ and  $\mu$  are nonzero ordinals.
If ${\mathcal K}(C (\lambda,   l_{p}), C(\xi, l_{q}))$ is isomorphic to  ${\mathcal K}
(C(\mu, l_{p}),  C(\eta, l_{q})),$ 
then    $\lambda_0= \mu_0$. Moreover, if  ${|\lambda|}< m_{r}$
then $\psi(\lambda,\xi)
= \psi(\mu,\eta).$
\end{theorem}

\begin{proof}   Assume that ${\mathcal K}(C (\lambda,   l_{p}), C(\xi, l_{q})) $ is isomorphic to 
${\mathcal K}
(C(\mu, l_{p}),  C(\eta, l_{q}))$ for some nonzero finite ordinals  $\xi,$ $\eta$, $\lambda$ and  $\mu$. 
Without loss of generality we can assume   $|\lambda| \leq |\mu|$. Let
$p'$ be the real number satisfying  $1/p + 1/p' =1$. We know that
$C(\beta, l_{p})$ has the approximation property for every ordinal $\beta.$
So according to  \cite[Proposition
5.3]{DF} we see   that

$$ {\mathcal K}\left(C (\lambda,   l_{p}), C(\xi, l_{q})\right )            \sim             l_{1}(\lambda, l_{p'}) {\hat{\hat\otimes}}  C(\xi, l_{q}) \sim   C( \xi, l_{1}(\lambda, l_{p'}) {\hat{\hat\otimes}}  l_{q}), \eqno (3)$$ 
and 
$${\mathcal K} (C(\mu, l_{p}),  C(\eta, l_{q}))  \sim     l_{1}(\mu, l_{p'}) {\hat{\hat\otimes}}  C(\eta, l_{q})  \sim   C( \eta, l_{1}(\mu, l_{p'}) {\hat{\hat\otimes}}  l_{q})              \eqno(4)$$ 
Consequently
$$l_{1}(\mu)  \hookrightarrow  {\mathcal K}(C(\mu, l_{p}),  C(\eta, l_{q}))    \sim    C( \xi, l_{1}(\lambda, l_{p'}) {\hat{\hat\otimes}}  l_{q}).               \eqno (5)$$
Since  $l_{1}(\mu)$ contains no subspace isomorphic to  $c_{0}$ and $l_{1}(\lambda, l_{p'}) {\hat{\hat\otimes}} l_{q}$ is isomorphic to its square, we deduce by (5) and 
\cite[Theorem 2.3] {G2} that 
$$l_{1}(\mu) \hookrightarrow l_{1}(\lambda, l_{p'}) {\hat{\hat\otimes}} l_{q}. \eqno (6)$$
Therefore if $\lambda< \omega$ and $\mu \ge \omega$, then by (6) we
would conclude that   $l_{p'} {\hat{\hat\otimes}} l_{q}$ contains a copy of
$l_{1}$, which is absurd by \cite[Theorem 3.1]{S1}. Hence either $\lambda
\mu < \omega$ and $\lambda_0=\mu_0$
or $\lambda \geq \omega$ and $\mu \geq \omega$. In the
last case,  observe that the  density  character of  $l_{1}(\mu)$ is
$|\mu|$ and  the density character  of  $l_{1}(\lambda, l_{p'})
{\hat{\hat\otimes}} l_{q}$ is $|\lambda|$. Thus, it follows from
(6) that $|\mu| \leq |\lambda|$. So  $|\lambda|= |\mu|$. Thus we have that
$\lambda_0=\mu_0.$

If $\xi<\omega$ and $\lambda_0\ge \omega$, then by Corollary \ref{notiso},
$\eta<\omega$, and $\psi(\lambda,\xi)=\psi(\mu,\eta).$
If $\lambda_0=\mu_0=1$, then if $\xi$ or $\eta$ is finite, we can replace it
by $\omega,$ \cite[Theorem 5.3]{AG}. (Observe that
$\psi(m,n)=\psi(m,\omega)$ if $m,n<\omega.$) Thus we may assume for the
remaining cases that $\xi$
and $\eta$ are infinite and that $\lambda=\lambda_0=\mu.$

Next,  let $$X=l_{1}(\lambda_0, l_{p'}) {\hat{\hat\otimes}}
l_{q}=l_{1}(\lambda, l_{p'}) {\hat{\hat\otimes}}
l_{q}=l_{1}(\mu, l_{p'}) {\hat{\hat\otimes}}
l_{q}.$$
By Corollary \ref{class_F}, $X$ has the Mazur Property.
Hence by
(3) and  (4)  we infer that
$$C(\xi, X) \sim     {\mathcal K}(C (\lambda,	l_{p}), C(\xi, l_{q}))
\sim {\mathcal K} (C(\mu, l_{p}),  C(\eta, l_{q}))  \sim C(\eta,
X). \eqno(7)$$

Thus by (1) of  Theorem \ref{Bu} we conclude that  $|\xi|
= |\eta|.$ The other assertions of that theorem allow us to complete the
argument once we deal with Theorem \ref{Bu}(\ref{regular_ord})(a). 

Suppose that $|\xi|=\xi_0$ is an uncountable regular cardinal. Because
$X$ is isomorphic to its square, if the ordinal quotient $\xi'$ is finite,
$C(\xi,X)$ is isomorphic to $C(\xi_0,X).$ If
$\lambda_0=1$ and $|\xi'|=\omega$,
then
$$C(\xi_0\omega,X)\sim
C(\xi_0,C(\omega,X))\sim C(\xi_0,X).$$ However, if $|\xi'|=\omega$ and
$\lambda\ge \omega,$ then
by Corollary \ref{notiso} $C(\omega,X)$ is not isomorphic to $X$. Thus
$C(\xi_0\omega,X)$ is not isomorphic to  $C(\xi_0,X).$
Thus $\psi(\lambda,\xi)=\psi(\mu,\eta)$ in all cases.

\end{proof}

The results claimed in the abstract are immediate consequences.

\begin{theorem}  \label{iso} Suppose that  $1 <p \leq q < \infty$  and
$\xi$,  $\eta$,  $\lambda$ and  $\mu$  are infinite ordinals with 
$|\lambda|=|\mu|<m_{r}$.
Then the following statements   are equivalent

\begin{enumerate}
\item[(a)] ${\mathcal K}(C(\lambda, l_{p}), C(\xi, l_{q})) $ is isomorphic to  ${\mathcal K}(C(\mu, l_{p}), C( \eta, l_{q})) .$ 

\item  [(b)] $C(\xi)$ is  isomorphic to $C(\eta)$ or there exists a finite or an
uncountable  regular  ordinal  $\alpha$  and $ 1 \leq m, n  <
\omega$ such that $C(\xi)$ is isomorphic to $C(\alpha m)$ and $C(\eta)$
is isomorphic to $C(\alpha n)$. 
\end{enumerate}
\end{theorem}

Finally, the next theorem completes the isomorphic classification, up to isomorphism, of the spaces of compact operators on $C(\xi, l_{p})$ spaces, with $1<p< \infty$.

\begin{theorem} \label{t2}Suppose that  $1 <p \leq q < \infty$  and
$\xi$,  $\eta$,  $\lambda$ and  $\mu$  are ordinals with
$\omega \le \xi \leq \eta$, $|\xi|=|\eta|$ and $\lambda   \mu < \omega$. Then the
following statements  are equivalent

\begin{enumerate}
\item[(a)] ${\mathcal K}(C(\lambda, l_{p}), C(\xi, l_{q})) $ is isomorphic to  ${\mathcal K}(C(\mu, l_{p}), C(\eta, l_{q})) .$ 

\item  [(b)] $C(\xi)$ is  isomorphic to $C(\eta)$ or there exists an uncountable  initial regular  ordinal  $\alpha$  and $ 1 \leq m, n  \leq  \omega$ such that $C( \xi)$ is isomorphic to $C(\alpha m)$ and $C(\eta)$ is isomorphic to $C(\alpha n)$.
\end{enumerate}
\end{theorem}






In view of Theorems \ref{iso} and \ref{t2}, the following question arises naturally.
\begin{problem}  Does the above  isomorphic classification of the spaces of compact operator  ${\mathcal K}(C (\lambda,   l_{p}), C(\xi, l_{q}))$, with $\lambda \geq \omega$,  remain true in ZFC?
\end{problem}

\bibliographystyle{amsplain}

\begin{thebibliography}{10}

\bibitem {AG}  D. E. Alspach, E. M.  Galego,  \textit {Geometry of the
Banach spaces $C(\beta {\mathbb N}  \times K, X)$ for compact metric spaces
K}. Studia Math. {\textbf 207} (2011),  2, 153-180.
		
		
\bibitem {BP}  C. Bessaga, A. Pe\l czy\'nski, \textit {Spaces of 
continuous
functions IV,} Studia Math. {\textbf XIX} (1960), 53-61.



\bibitem{CP} P.G. Casazza, \textit {Approximation properties}. Handbook of 
the
geometry of Banach spaces I. North-Holland Publishing Co., Amsterdam. 
(2001), 271-316.

\bibitem {DF} A. Defant and K. Floret, \textit {Tensor norms and operators 
ideals}, Math. Studies 176, North-Holland, Amsterdam (1993).



\bibitem {D} J. Diestel,  J.J.J.R. Uhl, \textit {Vector Measures,}
Mathematical Surveys 15, Amer. Math. Soc.  Providence, RI (1977).

\bibitem {E} G. A. Edgar,  \textit {Measurability in a Banach space II,}
Indiana Univ. Math. {\textbf 28} (1977), 559-579.


\bibitem {GP} R. J. Gardner, W. F.  Pfeffer. \textit {Borel measures.  
Handbook of set-theoretic topology}.   North-Holland, Amsterdam,
(1984), 961-1043.




\bibitem {G2} E. M. Galego, \textit {On subspaces and quotients of Banach
spaces $C(K,X)$}.  Monatsh. Math.  {\textbf 136}  (2002),  2, 87-97.

\bibitem {G21} E. M. Galego, \textit {An isomorphic classification of $C({\bf 
2}^{\mathfrak{m}} \times [0, 
\alpha])$ spaces}. Bull. Pol. Acad. Sci. Math.  {\textbf 57} (2009), 3-4, 279-287


\bibitem {G3} E. M. Galego,  \textit {On isomorphic classifications of 
space of 
compact
operators}. Proc. Amer. Math. Soc. {\textbf 137} (2009), 10, 3335-3342.

\bibitem {G4} E. M. Galego,  Complete isomorphic classifications of some spaces of compact operators.  Proc. Amer. Math. Soc. {\textbf 138} (2010),  2, 725-736. 

\bibitem {G5} E. M. Galego,  On spaces of compact operators on $C(K,X)$ spaces. Proc. Amer. Math. Soc. {\textbf 139} (2011),  4, 1383-1386.

\bibitem {GO} S.P. Gul’ko, A.V. Os’kin, Isomorphic classiﬁcation of spaces of continuous functions on totally ordered bicompacta, Funct. Anal. Appl. 9 (1) (1975) 56–57.

\bibitem {J} T. Jech, Set Theory, Academic Press. New York, San Francisco, London. (1978).




\bibitem {JL}   W. B. Johnson, J. Lindenstrauss,  \textit {Handbook of the
geometry of Banach spaces}. North-Holland Publishing Co., Amsterdam. 
(2001),
1-84.



\bibitem {KM} A. Kanamori, A., M. Magidor,, \textit
{The evolution of large cardinal axioms in set theory.}  Higher set theory
(Proc. Conf., Math. Forschungsinst., Oberwolfach, 1977),  
Lecture Notes in Math., 669, Springer, Berlin. (1978), 99-275.

\bibitem {Ka} T. Kappeler, \textit {Banach spaces with condition of
Mazur's},
Math. Z. {\textbf 191} (1986), 623-631.

\bibitem {Kis} S.V. Kislyakov, Classiﬁcation of spaces of continuous functions of ordinals, Siberian Math. J. 16 (2) (1975) 226–231.


\bibitem {L}  D. Leung, \textit {Banach spaces with Mazur's property},
Glasgow  
Math. J. {\textbf 33} (1991), 51-54.



\bibitem{LT} J. Lindenstrauss, L.  Tzafriri, \textit {Classical Banach
spaces
I. Sequence Spaces}, Springer-Verlag, Berlin-New York. (1977).


\bibitem {R}  Rosenthal, H. P. \textit {On relatively disjoint families of 
measures, with some applications to Banach space theory},  Studia Math.  
{\textbf 37}  1970 13-36


\bibitem{S1} C. Samuel, \textit{Sur la reproductibilite des espaces 
$l_{p}$}, Math. Scand. {\textbf 45} (1979) 103-117.

\bibitem{S2} C. Samuel,  \textit{Sur les sous-espaces de  $l_{p}
{\hat{\hat\otimes}}  l_{q}$ }     (French) Math. Scand. {\textbf 47}  (1980),  2, 247-250.


\bibitem {W} A.  Wilansky,  \textit {Mazur spaces},  Internat. J. Math. 
Math.  
Sci. {\textbf 4} (1981), 39-53.


\end{thebibliography}

\end{document}